\pdfoutput=1
\documentclass[11pt]{amsart}
\usepackage[marginratio=1:1]{geometry}
\usepackage{amssymb}
\usepackage[colorlinks,linkcolor=blue,anchorcolor=blue,
citecolor=blue]{hyperref}

\theoremstyle{definition}
\newtheorem{theorem}{Theorem}
\newtheorem{prop}[theorem]{Proposition}
\newtheorem{cor}[theorem]{Corollary}
\newtheorem{definition}[theorem]{Definition}
\newtheorem{remark}[theorem]{Remark}
\newtheorem{lemma}[theorem]{Lemma}
\newtheorem{example}[theorem]{Example}
\newcommand{\CC}{\mathbb{C}}

\newcommand{\QQ}{\mathbb{Q}}
\newcommand{\ZZ}{\mathbb{Z}}

\newcommand{\ps}{\mathcal{P}}

\renewcommand{\AA}{\mathbb{A}}
\newcommand{\V}{\mathcal{V}}
\newcommand{\SymGp}{{{\mathfrak{S}_n}}}
\newcommand{\good}{good}
\newcommand{\bad}{bad}

\renewcommand{\geq}{\geqslant}
\renewcommand{\leq}{\leqslant}

\title{Degrees of regular sequences with a symmetric group action}
\author{Federico Galetto, Anthony V.~Geramita, David L.~Wehlau}
\date{\today}

\begin{document}

\begin{abstract}
  We consider ideals in a polynomial ring that are generated by
  regular sequences of homogeneous polynomials and are stable under
  the action of the symmetric group permuting the variables.  In
  previous work, we determined the possible isomorphism types for
  these ideals. Following up on that work, we now analyze the possible
  degrees of the elements in such regular sequences. For each case of
  our classification, we provide some criteria guaranteeing the
  existence of regular sequences in certain degrees.
\end{abstract}

\maketitle
\section{Introduction}
Consider the graded polynomial ring $R = \CC [x_1,x_2,\dots,x_n]$.  A
set of $n$ homogeneous polynomials $f_1,f_2,\dots,f_n$ is a maximal
regular sequence in $R$ if the only common zero of these $n$
polynomials is the point $(0,0,\dots,0)$.  A sequence
$g_1,g_2,\dots,g_t$ is a regular sequence in $R$ if it can be extended
to a maximal regular sequence in $R$.
  
We suppose that $G$ is a group acting linearly on $R$ via an action
which preserves the grading.  The subring
$R^G := \{f \in R : \forall \sigma \in G, \sigma\cdot f = f\}$ is
called the ring of invariants.  There has been some interest in
determining the degrees $(d_1,d_2,\dots,d_t)$ for which there exists a
regular sequence in $R^G$ with $\deg (f_i) = d_i$.  Dixmier
\cite{Dixmier} made a conjecture concerning this question for the
classical case of the action of $\text{SL}(2,\CC)$ on an irreducible
representation.  This conjecture has attracted some attention.  See
for example \cite{L-P,DD,BBS}.  Recently, a few authors have taken up
this question for the natural action of the symmetric group on $R$.
See \cite{CKW,Chen,K-M}.
   
We consider a more general question.  Our goal is to determine the
degrees of a maximal regular sequence $f_1,f_2,\dots,f_n$ in $R$ such
that the ideal $I:=(f_1,f_2,\dots,f_n)$ is stable under the group
action.  This is equivalent to the artinian quotient algebra $R/I$
inheriting the action of the group.
   
We will also restrict our attention to the natural action of the
symmetric group $\SymGp$ permuting the variables.  In our earlier
paper \cite{SCI1}, it is shown that there are four possible
representation types for the action of $\SymGp$ on $I$ (the notation
follows that of \cite{Sagan}):
\begin{enumerate}
\item the trivial representation $S^{(n)}$, given by all $f_i$ being
  symmetric polynomials;
\item the alternating representation $S^{(1^n)}$, given by one
  alternating polynomial, together with up to $n-1$ symmetric
  polynomials;
\item the standard representation $S^{(n-1,1)}$, possibly together
  with one symmetric polynomial;
\item the representation $S^{(2,2)}$, together with up to two
  symmetric polynomials (this only occurs when $n=4$).
\end{enumerate}
Our earlier paper shows examples of regular sequences corresponding to
all four cases, but does not address the question of how ``often''
such regular sequences can appear or, more precisely, in what degrees
they can be realized. Here we give explicit answers showing in which
degrees it is possible to find a regular sequence for each of the
above four representation types for $n \leq 4$.  We also derive a
number of results for general values of $n$.

Note also that our results relating to the first case above, actually
apply to the degrees of regular sequences of homogeneous polynomials
in the polynomial ring $\CC [y_1,\ldots,y_n]$, with the non-standard
grading given by $\deg (y_i)=i$. This case corresponds geometrically
to the homogeneous coordinate ring of a weighted projective space.

\section{Regular sequences of symmetric polynomials}
We consider the polynomial ring $R = \CC [x_1,x_2,\dots,x_n]$ in $n$
indeterminates equipped with the standard grading.  The symmetric
group $\SymGp$ acts naturally on $R$ by permuting the variables. It is
well known that the invariant subring $R^\SymGp$ can be identified
with the subalgebra $\CC [e_1,e_2,\dots,e_n]$ generated by the
elementary symmetric polynomials \cite{Goodman}. In particular,
$R^\SymGp$ is a polynomial ring equipped with the non-standard grading
$\deg (e_i) = i$.

\subsection{Degree sequences}
We are concerned with the degrees of elements of homogeneous regular sequences in $R^\SymGp$.

\begin{definition}
  Let $(d_1,d_2,\dots,d_n)$ be an (unordered) sequence of $n$ positive integers.  If there exists a homogeneous regular
  sequence $f_1,f_2,\dots,f_n \in R^\SymGp$ with $\deg(f_i)=d_i$ then we say that $(d_1,d_2,\dots,d_n)$ is a \em{regular degree sequence}.
\end{definition}

\begin{prop}\label{beta condition}
  Suppose $(d_1,d_2,\dots,d_n)$ is a regular degree sequence.  For
  $i=2,3,\dots,n$ we define
  $\beta_i := \#\{1\leqslant j \leqslant n : i \mid d_j\}$.  Then
  \begin{equation}
    \beta_i \geq \bigg\lfloor \frac n i \bigg\rfloor \quad \text{for all } i = 1,2,\dots n  \tag{\raisebox{-2pt}{*}}  
  \end{equation}
  In particular, $n! \mid \prod_{j=1}^n d_j$.
\end{prop}

\begin{proof}
  If $ (d_1,d_2,\dots,d_n)$ is a regular degree sequence, then there
  exists a homogeneous regular sequence $f_1,f_2,\dots,f_n$ in
  $R^\SymGp$ with $\deg(f_i)=d_i$. The graded subring
  $A=\CC[f_1,f_2,\dots,f_n]$ is a polynomial ring and $R^\SymGp$ is a free
  $A$-module:
  $R^\SymGp \cong \oplus_{\gamma \in \Gamma}\,A \!\cdot\!\gamma$ for some
  set of homogeneous elements $\Gamma \subset R^\SymGp$ \cite[Lemma
  6.4.13]{Bruns}.  Thus the Hilbert series of $R^\SymGp$ and $A$ are related
  by
  $$\mathcal{H}(R^\SymGp,t) = \sum_{\gamma \in \Gamma} t^{\deg(\gamma)} \mathcal{H}(A,t).$$
  Since $\mathcal{H}(R^\SymGp,t) = \prod_{i=1}^n (1-t^i)^{-1}$ and
  $\mathcal{H}(A,t) = \prod_{i=1}^n (1-t^{d_i})^{-1}$, we see that
  $$\prod_{i=1}^n \frac{1-t^{d_i}}{1-t^i} = \sum_{\gamma\in\Gamma} t^{\deg(\gamma)}$$
  is a non-negative integer polynomial.

  Working over $\QQ$, all the irreducible factors of $(1-t^d)$ are
  cyclotomic polynomials.  Specifically
  $(1-t^d) = \prod_{i \mid d}\Phi_i(t)$, where $\Phi_i$ denotes the
  $i^\text{th}$ cyclotomic polynomial.  Since
  $\#\{1 \leq j \leq n : i \mid j \} = \lfloor n / i \rfloor$, we see
  that $\prod_{i=1}^n \frac{1-t^{d_i}}{1-t^i}$ is an integer
  polynomial if and only if
  $ \beta_i \geq \lfloor n / i \rfloor \quad \text{for all } i =
  1,2,\dots n$.
    
  To prove the final assertion, we cancel the factors of $(1-t)$ from
  the numerator and denominator.  Thus
    $$\prod_{i=1}^n \frac{1+t+t^2+\dots+t^{d_i}}{1+t+t^2+\dots + t^i} = \sum_{\gamma\in\Gamma} t^{\deg(\gamma)}\ .$$
    Evaluating at $t=1$ we see that
    $\left(\prod_{i=1}^n d_i\right)/n! = |\Gamma| =$ rank of
    $R^\SymGp$ as an $A$-module.
\end{proof}

\begin{remark}
  Our inequality (*) was first observed by Conca, Krattenthaler and
  Watanabe for regular sequences of power sums \cite[Lemma 2.6
  (2)]{CKW}. The three authors also showed that the product of the
  $d_i$ is divisible by $n!$ \cite[Lemma 2.8]{CKW}.  This seems to be the first time
  the restriction that  $$\prod_{i=1}^n \frac{1-t^{d_i}}{1-t^i} = \sum_{\gamma\in\Gamma} t^{\deg(\gamma)}$$
  is a non-negative integer polynomial has been observed.

\end{remark}

Suppose $(d_1,d_2,\dots,d_n)$ is a regular degree sequence.  Since
$\oplus_{d \leq i} R^\SymGp_d \subset \CC[e_1,e_2,\dots,e_i]$ and
hence cannot contain a regular sequence with more than $i$ terms, we
deduce that $(d_1,d_2,\dots,d_n)$ must also satisfy the following
condition
\begin{equation*}
  \#\{j :  d_j \leq i \} \leq i  \quad \text{for all } i = 1,2,\dots n\ .  \tag{\dag}
\end{equation*}

\begin{definition}
  Let $(d_1,d_2,\dots,d_n)$ be an (unordered) sequence of $n$ positive
  integers.  We say that $(d_1,d_2,\dots,d_n)$ is {\em permissible} if
  it satisfies the two conditions $(*)$ and $(\dag)$.  Thus every
  regular degree sequence is permissible.
\end{definition}

Note that if there exists a {\em matching}, i.e., a permutation
$\pi \in \SymGp$ such that $i$ divides $d_{\pi(i)}$ for all
$i=1,2,\dots,n$ then $(d_1,d_2,\dots,d_n)$ is a regular degree
sequence as is shown by the regular sequence of polynomials
$(e_i)^{d_{\pi(i)}/i}$ for $i = 1,2, \dots, n$.

\subsection{Regular degree sequences for $n \leq 4$}

\begin{theorem}
  \label{srs}
  \hspace{2em} %Without this LaTeX misses this reference
  \begin{enumerate}
  \item \label{n is 2} For $n=2$, a degree sequence is regular if and
    only if it is permissible if and only if it satisfies $(*)$.
  \item \label{n is 3} For $n=3$, a degree sequence is regular if and
    only if it is permissible.
  \item \label{n is 4} For $n=4$, every permissible degree sequence
    except $(1,2,5,12\delta), (2,2,5,12\delta)$ and $(5,2,5,12\delta)$
    is regular.
  \end{enumerate}
\end{theorem}

\begin{proof}
  (\ref{n is 2}) If $(d_1,d_2)$ satisfies $(*)$ then at least one of
  $d_1$ or $d_2$ is even and so we have a matching.

  (\ref{n is 3}) Let $n=3$ and suppose that $(d_1,d_2,d_3)$ is
  permissible but has no matching.  Then, without loss of generality,
  6 divides $d_3$ while $d_1$ and $d_2$ are both odd numbers not
  divisible by 3 with $d_2 \geq d_1$.  Now condition $(\dag)$ implies
  that $d_2 \geq 2$ and thus $d_2 \geq 5$.  Thus we have a regular
  sequence $e_1^{d_1}, e_3 e_2^{(d_2-3)/2}, (e_2^3 + e_3^2)^{d_3/6}$
  with degrees $(d_1,d_2,d_3)$.

  (\ref{n is 4}) Let $n=4$ and suppose that $(d_1,d_2,d_3,d_4)$ is
  permissible but has no matching.  Note that
  $R^\SymGp_1 \oplus R^\SymGp_2 \oplus R^\SymGp_5$ is contained in the
  ideal generated by $e_1$ and $e_2$.  This implies that a regular
  degree sequence must satisfy $\#\{i : d_i \in \{1,2,5\} \} \leq 2$.
  This shows that the three permissible degree sequences
  $(1,5,2,12\delta), (2,5,2,12\delta), (5,5,2,12\delta)$ are not
  regular.

  Condition $(*)$ implies that two of the $d_i$ are even, one is
  divisible by 3 and one is divisible by 4.  Without loss of
  generality, $d_3$ and $d_4$ are both even, $d_4 = 4\delta$ is
  divisible by 4 and $d_2 \geq d_1$.  Since there is no matching,
  neither $d_1$ nor $d_2$ is divisible by 3.  Thus either $d_3$ is a
  multiple of 6 or $d_4$ is a multiple of 12.

  First we consider the case where $d_3 = 6 \beta$ is a multiple of 6.
  Since there is no matching both $d_1$ and $d_2$ are odd integers not
  divisible by 3.  Thus $(\dag)$ implies that $d_2 \geq 5$.  Therefore
  $e_1^{d_1}, e_3 e_2^{(d_2-3)/2}, (e_2^3 + e_3^2)^{\beta},
  e_4^{\delta}$ is a regular sequence of the required degrees.

  Thus we may suppose that $d_4 = 12 \delta$ is a multiple of 12.  Now
  we adjust our labelling as follows.  We suppose that $d_3$ is the
  largest of those elements of $\{d_1,d_2,d_3\}$ which are even.
  Further we assume that $d_2 \geq d_1$.  Furthermore, since there is
  no matching, 3 divides neither $d_1$ nor $d_2$.
      
  Since $d_2 \geq 2$ we may write $d_2 = 2p + 3q$ where $p$ and $q$
  are non-negative integers.  Suppose first that $d_3 \geq 4$ and
  define
  $$f :=
  \begin{cases}
    e_4^{d_3/4}, & \mbox{if } d_3 \equiv 0 \pmod 4;\\
    e_4^{(d_3-6)/4}(e_2^3+ e_3^2), & \mbox{if } d_3 \equiv 2 \pmod 4.
  \end{cases}
  $$
  Then $e_1^{d_1}, e_2^p e_3^q, f, (e_2^6 + e_3^4 + e_4^3)^{\delta}$
  is a regular sequence of degrees $(d_1,d_2,d_3,d_4)$.
     
  Finally we suppose that $d_3 = 2$.  Then either $d_2=2$ or $d_2$ is
  odd.  But we have seen that $(1,2,2,12\delta)$ and
  $(2,2,2,12\delta)$ are not regular degree sequences and thus $d_2$
  must be odd.  Since 3 does not divide $d_2$, we have $d_2\geq 5$.
  If $d_2 = 5$ then $d_1 \in \{1,2,5\}$ which is again not possible
  since $(1,5,2,12\delta)$, $(2,5,2,12\delta)$ and $(5,5,2,12\delta)$
  are not regular degree sequences.  Therefore $d_2 \geq 7$ if
  $d_3=2$.  Thus we may write $d_2 =3 p + 4q$.  Then
  $e_1^{d_1}, e_3^p e_4^q, e_2, (e_3^4+e_4^3)^{\delta}$ is a regular
  sequence of the required degrees.
\end{proof}

Table \ref{tab:symmetric} summarizes the regular sequences we have
found when $n=4$.
The first row corresponds to a matching.

\begin{table}[htb]
  \centering
  $\begin{array}{|l|l|l|l|l|l|l|l|}
     \hline
     \multicolumn{4}{|c|}{\rm \bf Degrees} & \multicolumn{4}{|c|}{\rm\bf Symmetric\ Polynomials}\\
     \hline
     \multicolumn{1}{|c|}{d_1} & \multicolumn{1}{|c|}{d_2} & \multicolumn{1}{|c|}{d_3} & \multicolumn{1}{|c|}{d_4} 
                                                                                       & \multicolumn{1}{|c|}{f_1} & \multicolumn{1}{|c|}{f_2} & \multicolumn{1}{|c|}{f_3} & \multicolumn{1}{|c|}{f_4}\\ \hline
     d_1&  3\beta & 2\gamma & 4\delta & e_1^{d_1} & e_3^{\beta} & e_2^{\gamma} & e_4^{\delta}\\      %matching 
     d_1 & d_2\geq 5  &6\gamma &4\delta& e_1^{d_1} & e_3 e_2^{(d_2-3)/2} & (e_2^3 + e_3^2)^{\gamma}& e_4^{\delta}\\
     d_1 & d_2 \geq 2 & 4\beta  & 12\delta & e_1^{d_1} & e_2^p e_3^q & e_4^\beta & (e_2^6+e_3^4+e_4^3)^{\delta}\\
     d_1 & d_2 \geq 2 & 4\beta+2 \geq 6 & 12\delta & e_1^{d_1} & e_2^p e_3^q & (e_2^3+e_3^2)e_4^{\beta-1} & (e_2^6+e_3^4+e_4^3)^{\delta}\\
     d_1 & d_2 \geq 7 & 2 & 12\delta & e_1^{d_1} & e_3^p e_4^q & e_2 & (e_3^4+e_4^3)^{\delta}\\ 
     \hline
   \end{array}$
   \bigskip
   \caption{Regular sequences of symmetric polynomials for $n=4$}\label{tab:symmetric}
\end{table}

Note that we have in fact proved the following result.
\begin{cor}\label{improved cor}
  Suppose that $d_2,d_3,d_4$ are three positive integers such that
  $4\mid d_4$, $d_3$ is even, and $3\mid d_2 d_3 d_4$.  Then there
  exist three symmetric polynomials $f_2,f_3,f_4$ (as given in
  Table~1) of degrees $d_2,d_3,d_4$ respectively such that
  $e_1,f_2,f_3,f_4$ is a regular sequence.
\end{cor}
  
\begin{remark}
  Note that the degree sequence $(2,5,2,12)$ (which is not regular)
  has the property that
  $$ \frac{(1-t^2)(1-t^5)(1-t^2)(1-t^{12})}{(1-t)(1-t^2)(1-t^3)(1-t^4)} = 1 + t + t^3 + 2t^4 + 2t^7 +t^8 + t^{10} + t^{11}$$
  is a non-negative integer polynomial.
\end{remark}

For larger values of $n$ little is known. The following statement was
proved in \cite[Prop.~2.9]{CKW} using sequences of power sums and
homogeneous symmetric polynomials.

\begin{prop}\label{consecutive}
  For every positive integer $a$ the sequence of consecutive degrees
  $(a, a+1, a+2, \dots, a+n-1)$ is a regular degree sequence.
\end{prop}

\subsection{Regular sequences with an alternating polynomial}
\label{sec:alt-rep}

A polynomial $f\in R$ is said to be \emph{alternating} if, for all
$\sigma \in \SymGp$, $\sigma f = \pm f$, depending on the sign of
$\sigma$.  As an example, the Vandermonde determinant
$$\Delta := \det 
\begin{bmatrix}
  1 &x_1 &x_1^2 &\dots &x_1^{n-1}\\
  1 &x_2 &x_2^2 &\dots &x_2^{n-1}\\
  \vdots &\vdots &\vdots & &\vdots\\
  1 &x_n &x_n^2 &\dots &x_n^{n-1}
\end{bmatrix} = \prod_{1\leq i<j \leq n} (x_j-x_i) \in R$$ is clearly
alternating. In fact, every homogeneous alternating polynomial in $R$
is divisible by $\Delta$, the quotient being a homogeneous symmetric
polynomial.

As noted in \cite[Prop.~2.5]{SCI1}, there exist homogeneous regular
sequences $f_1,f_2,\dots,f_t,g\Delta$ in $R$ with $f_1,f_2,\dots,f_t$
and $g$ symmetric polynomials. These sequences are closely related to
sequences of symmetric polynomials.

\begin{lemma}\label{lem:reg_seq_factors}
  Let $f_1,f_2,\dots,f_t,g,h\in R$ be homogeneous polynomials. Then
  the sequence $f_1,f_2,\dots,f_t,g h$ is regular if and only if both
  $f_1,f_2,\dots,f_t,g$ and $f_1,f_2,\dots,f_t,h$ are regular.
\end{lemma}
\begin{proof}
  Suppose $f_1,f_2,\dots,f_t$ form a regular sequence. Then $g h$ is
  not a zero-divisor modulo $(f_1,f_2,\dots,f_t)$ if and only if both
  $g$ and $h$ are not zero-divisors modulo $(f_1,f_2,\dots,f_t)$.
\end{proof}

The following is an immediate consequence of Lemma
\ref{lem:reg_seq_factors}.

\begin{prop}\label{reg_seq_alt}
  Let $f_1,f_2,\dots,f_t,g\in R$ be homogeneous symmetric polynomials. The
  sequence $f_1,f_2,\dots,f_t,g\Delta$ is regular if and only if both
  $f_1,f_2,\dots,f_t,g$ and $f_1,f_2,\dots,f_t,\Delta$ are regular.
\end{prop}

Proposition \ref{reg_seq_alt} allows to rule out existence of regular
sequences of certain degrees that contain an alternating polynomial.

\begin{example}
  For $n=4$, $\Delta$ has degree 6. By Theorem \ref{srs}
  (\ref{n is 4}), there is no regular sequence of homogeneous
  symmetric polynomials $f_1,f_2,f_3,g$ of degrees
  $1,2,5,12\delta$. Therefore, Proposition \ref{reg_seq_alt} implies
  there is no regular sequence $f_1,f_2,f_3,g\Delta$ of degrees
  $1,2,5,12\delta + 6$.
\end{example}

\begin{remark}
  The polynomial $\Delta^{2k}$ is symmetric for all positive integers
  $k$. Moreover, the sequence $f_1,f_2,\dots,f_t,\Delta$ is regular if
  and only if $f_1,f_2,\dots,f_t,\Delta^{2k}$ is regular
  (cf. \cite[Cor.~17.8 a]{Eisenbud}). As a consequence, we can exclude
  the existence of regular sequences in certain degrees. For example,
  there is no regular sequence of homogeneous polynomials
  $f_1,f_2,f_3,\Delta$ with $f_1,f_2,f_3$ symmetric of degrees $1,2,5$
  because $f_1,f_2,f_3,\Delta^2$ would violate Theorem
  \ref{srs} (\ref{n is 4}).
\end{remark}

\section{Regular sequences and the standard representation}
\label{sec:std-rep}

We begin this section by recalling some basic facts about the
representation theory of the symmetric group $\SymGp$ over a field of
characteristic zero. We refer the reader to \cite[Ch.~2]{Sagan} for
the details.

We write $\lambda \vdash a$ to denote that
$\lambda=(\lambda_1,\lambda_2,\dots,\lambda_r)$ is a partition of the
integer $a$, %with no part exceeding $n$,
i.e., that $\lambda_1 + \lambda_2 + \dots + \lambda_r=a$ and
$\lambda_1 \geq \lambda_2 \geq \dots \geq \lambda_r >
0$. % and $\lambda_t \leq n$ for all $t$.
The irreducible representations of $\SymGp$ are in bijection with the
partitions of $n$; for $\lambda \vdash n$, we denote by $S^\lambda$
the corresponding irreducible.  Every finite dimensional
representation of $\SymGp$ decomposes into a direct sum of copies of
the $S^\lambda$.

The irreducible representation $S^{(n-1,1)}$ of $\SymGp$ is often
called the \emph{standard representation}. It can be described as the
$\SymGp$-stable complement of the subspace spanned by $e_1$ inside the
representation $R_1 = \langle x_1,x_2,\dots,x_n \rangle$. The
polynomials $x_1-x_n,x_2-x_n,\dots,x_{n-1}-x_n$ give an explicit basis
of the complement.

Let $\mathfrak{m} = (x_1,x_2,\dots,x_n)$ be the irrelevant maximal
ideal of $R$.  In this section, we study homogeneous regular sequences
$f_1,f_2,\dots,f_t \in R$ such that the ideal $I=(f_1,f_2,\dots,f_t)$
is stable under the action of $\SymGp$ and $I / \mathfrak{m}I$
contains a copy of the standard representation.  As shown in
\cite[Prop.~2.5]{SCI1}, there are two possibilities:
$I / \mathfrak{m}I \cong S^{(n-1,1)}$ or
$I / \mathfrak{m}I \cong S^{(n-1,1)} \oplus S^{(n)}$, where $S^{(n)}$
is the one-dimensional trivial representation.

\subsection{Regular sequences of type $S^{(n-1,1)}$}\label{reduced rep}

Here we prove the existence of regular sequences of type $S^{(n-1,1)}$
in every positive degree.

Let $\V_d \subset \AA^n$ denote the affine variety cut out by the
$x_1^d-x_n^d,x_2^d-x_n^d,\dots,x_{n-1}^d-x_n^d$ and $x_n=1$; i.e.,
$$\V_d = \{(z_1,z_2,\dots,z_n) \in \AA^n : z_i^d=1, z_n=1\}.$$ 

\begin{theorem}
  \label{red_rep}
  Let $d$ be a positive integer. The polynomials
  $x_1^d - x_n^d,x_2^d - x_n^d,\dots,x_{n-1}^d - x_n^d$ form a regular
  sequence of type $S^{(n-1,1)}$.
\end{theorem}
 
\begin{proof}
  The polynomials in question form a basis of the $\SymGp$-stable
  complement of the one-dimensional invariant subspace spanned by
  $x_1^d + x_2^d + \dots + x_n^d$ inside
  $\langle x_1^d, x_2^d, \dots, x_n^d\rangle$.  It is clear from the
  comments at the beginning of the section that this complement is
  isomorphic to $S^{(n-1,1)}$.

  To prove $x_1^d - x_n^d,x_2^d - x_n^d,\dots,x_{n-1}^d - x_n^d$ form
  a regular sequence, we extend it by adding the polynomial $x_n^d$.
  It is clear that the two ideals
  $(x_1^d - x_n^d,x_2^d - x_n^d,\dots,x_{n-1}^d - x_n^d,x_n^d)$ and
  $(x_1^d ,x_2^d,\dots,x_{n-1}^d ,x_n^d)$ are equal and that the
  latter is generated by a regular sequence.  Thus the extended
  sequence, and so also the original, is a regular sequence.
\end{proof}  
  
\subsection{Regular sequences of type $S^{(n-1,1)} \oplus S^{(n)}$}

Let $I\subseteq R$ be an $\SymGp$-stable homogeneous ideal such that
$I/\mathfrak{m}I \cong S^{(n-1,1)} \oplus S^{(n)}$. Then $I$ admits a
generating set $g_1,g_2,\dots,g_{n-1},f$ such that:
\begin{itemize}
\item $\deg (g_i) = d$ for $i=1,2,\dots,n-1$ and the vector space
  spanned by $g_1,g_2,\dots,g_{n-1}$ is a representation of $\SymGp$
  isomorphic to $S^{(n-1,1)}$;
\item $\deg (f) = a$ and $f \in R^\SymGp$.
\end{itemize}
We are interested in understanding the possible choices of degrees $d$
and $a$ for which such an ideal $I$ can be generated by a regular
sequence.  For simplicity, we restrict to the case
$g_i = x_i^d - x_n^d$ for $i=1,2,\dots,n-1$. This is the instance of
regular sequence described in Theorem \ref{red_rep}. Therefore our
main question becomes: when can a symmetric polynomial $f$ of degree
$a$ be chosen so that
$x_1^d - x_n^d, x_2^d - x_n^d, \dots, x_{n-1}^d - x_n^d, f$ is a
regular sequence?

\begin{definition}\label{good_bad}
  Let $n,d,a$ be three positive integers.  We say the triple $(n,d,a)$
  is {\em \good} if there exists $f \in R_a^{\SymGp}$ such that
  $x_1^d-x_n^d,x_2^d-x_n^d,\dots,x_{n-1}^d-x_n^d,f$ is a regular
  sequence.  Otherwise $(n,d,a)$ is called {\em \bad}.
\end{definition}

\begin{remark}
  \label{rem:M2_rem}
  Clearly, if $(n,d,a)$ is \good, then there exists a regular sequence
  of type $S^{(n-1,1)} \oplus S^{(n)}$ with $S^{(n-1,1)}$ in degree
  $d$ and $S^{(n)}$ in degree $a$.  However, the converse is not true
  in general.  For example, the triple $(5,6,1)$ is \bad{} because
  $x_1^6-x_5^6,x_2^6-x_5^6,x_3^6-x_5^6,x_4^6-x_5^6,e_1$ is not a
  regular sequence.  However, if we set
  $g_i = \sum_{j=2}^5 e_j (x_i^{6-j} - x_5^{6-j})$ for $i=1,2,3,4$,
  then $g_1,g_2,g_3,g_4,e_1$ is a regular sequence.  The assertions
  about these sequences of polynomials can be verified computationally
  using the software Macaulay2 \cite{M2}, and the code provided in
  Appendix \ref{sec:macaulay2-code}.
\end{remark}

Observe that, if $f \in R$ is homogeneous, then
$x_1^d-x_n^d,x_2^d-x_n^d,\dots,x_{n-1}^d-x_n^d,f$ is a regular
sequence if and only if $f$ does not vanish on $\V_d$.

For a positive integer $a$, the \emph{power sum}
$\ps_a = x_1^a + x_2^a + \dots + x_n^a$ is a homogeneous symmetric
polynomial of degree $a$. Furthermore, given a partition
$\lambda = (\lambda_1, \lambda_2, \dots, \lambda_r)$ of $a$, we write
$\ps_\lambda$ for the symmetric polynomial
$\prod_{t=1}^r \ps_{\lambda_t}$ of degree $a$.  The set of
$\ps_\lambda$ with
$\lambda = (\lambda_1, \lambda_2, \dots, \lambda_r)$ a partition of
$a$ whose parts $\lambda_i$ do not exceed $n$ is a basis of
$R^\SymGp_a$ as a complex vector space
(cf. \cite[Prop.~7.8.2]{Stanley}).

\begin{lemma}\label{power sums suffice}
  The triple $(n,d,a)$ is \bad{} if and only if there exists a point
  $Q \in \V_d$ such that $\ps_\lambda(Q)=0$ for every partition
  $\lambda \vdash a$.
\end{lemma}

\begin{proof}
  If such a point $Q$ exists, then it is clear that $(n,d,a)$ is \bad.
  Suppose then that $(n,d,a)$ is \bad{}. Enumerate the partitions
  $\lambda \vdash a$ whose parts do not exceed $n$ and denote them by
  $\lambda^{(1)}, \lambda^{(2)}, \dots, \lambda^{(t)}$. Introduce the
  homogeneous symmetric polynomial
  $$f := \sum_{i=1}^t \pi^i \ps_{\lambda^{(i)}}$$
  of degree $a$. Since $(n,d,a)$ is \bad{}, there exists $Q\in \V_d$
  such that
  $$0 = f(Q) = \sum_{i=1}^t \pi^i \ps_{\lambda^{(i)}} (Q).$$
  Since the coordinates of $Q$ are algebraic numbers,
  $\ps_{\lambda^{(i)}} (Q)$ is algebraic for all $i=1,2,\dots,t$.
  Then $f(Q)=0$ implies $\ps_{\lambda^{(i)}} (Q) = 0$ for all
  $i=1,2,\dots,t$ because $\pi$ is transcendental. The result follows.
\end{proof}

The following is an immediate consequence of Lemma \ref{power sums
  suffice}.

\begin{cor}\label{single Q}
  The triple $(n,d,a)$ is \bad{} if and only if there exists a point
  $Q \in \V_d$ such that $f(Q)=0$ for every $f \in R_a^{\SymGp}$.
\end{cor}

Lemma \ref{power sums suffice} suggests it might be useful to
understand the vanishing of power sums at roots of unity. The
following result is due to Lam and Leung \cite[Thm.~5.2]{vanishing
  sums}.

\begin{theorem}\label{LL}
  Let $d$ be a positive integer and let $\Gamma(d)$ denote the
  numerical semi-group generated by the prime divisors of $d$.  Then
  there exist $d^\text{th}$ roots of unity $z_1,z_2,\dots,z_n$ (not
  necessarily distinct) such that $z_1+z_2+\dots+z_n=0$ if and only
  if $n \in \Gamma(d)$.
\end{theorem}

Note that $\Gamma(1):=\{0\}$ here.

\begin{cor}\label{LL_cor}
  Let $a,d$ be positive integers and let $g:=\gcd(a,d)$.  Then
  there exist $d^\text{th}$ roots of unity $z_1,z_2,\dots,z_n$ (not
  necessarily distinct) such that $\ps_a(z_1,z_2,\dots,z_n)=0$ if and
  only if $n \in \Gamma(d/g)$.
\end{cor}
 
\begin{proof}
  Assume there exist $d^\text{th}$ roots of unity $z_1,z_2,\dots,z_n$
  such that $\ps_a(z_1,z_2,\dots,z_n)=0$. Note that $z_i^a$ is a
  $(d/g)^\text{th}$ root of unity. Then
 $$z_1^a+z_2^a+\dots+z_n^a =\ps_a(z_1,z_2,\dots,z_n) =0$$
 implies $n \in \Gamma (d/g)$ by Theorem \ref{LL}.

 Conversely, assume $n \in \Gamma(d/g)$. Then Theorem \ref{LL} implies
 the existence of $(d/g)^\text{th}$ roots of unity $w_1,w_2,\dots,w_n$
 such that $w_1+w_2+\dots+w_n=0$. Since $g = \gcd(a,d)$, we have
 $1 = \gcd(a,d/g)$. By Bezout's identity \cite[Prop.~5.1]{Lang},
 there exist integers $u,v$ such that $au+(d/g)v=1$. Note that
 $z_i = w_i^u$ is a $d^\text{th}$ root of unity. Therefore we get
 $$0 = \sum_{i=1}^n w_i = \sum_{i=1}^n w_i^{au+(d/g)v} =
 \sum_{i=1}^n (w_i^u)^a (w_i^{d/g})^v = \ps_a (z_1,z_2,\dots,z_n).$$
\end{proof}

\begin{remark} \label{Galois remark}
  Let $\zeta_d$ be a primitive $d^{\rm th}$ root of unity. The Galois
  group of the cyclotomic field $\QQ (\zeta_d)$ is isomorphic to
  $(\ZZ/d\ZZ)^\times$, the group of units modulo $d$. An element of
  $(\ZZ/d\ZZ)^\times$ is represented by the class of an integer $s$
  coprime to $d$. Let $\gamma_s$ denote the corresponding Galois
  automorphism of $\QQ (\zeta_d)$, which is defined by fixing $\QQ$
  and sending $\zeta_d$ to $\zeta_d^s$. If $z$ is a $d^{\rm th}$ root
  of unity, then $z$ is a power of $\zeta_d$, therefore
  $\gamma_s (z) = z^s$.

  Now let $Q = (z_1,z_2,\dots,z_n)\in \V_d$. We have that
  \begin{equation*}
    \begin{split}
      \ps_s (Q) &= z_1^s + z_2^s + \dots + z_n^s
      =\gamma_s (z_1) + \gamma_s (z_2) + \dots + \gamma_s (z_n) =\\
      &=\gamma_s (z_1 + z_2 + \dots + z_n) =\gamma_s (\ps_1 (Q)).
    \end{split}
  \end{equation*}
  Therefore $\ps_s (Q) = 0$ if and only if $\ps_1 (Q) = 0$.
\end{remark}

\subsection{Numerical criteria for good and bad triples}
\label{sec:numerical-criteria}

Throughout the rest of this section $(n,d,a)$ is intended to be a
triple of positive integers. We present criteria to decide whether
$(n,d,a)$ is \good{} or \bad{} in the sense of Definition
\ref{good_bad}.
 
\begin{prop}
  \label{n_notin_Gamma}
  Let $g:=\gcd(a,d)$.  If
  $n \notin \Gamma(d/g)$, then $(n,d,a)$ is \good.
  In particular, if $n \notin \Gamma(d)$, then $(n,d,a)$ is \good{}
  for every $a$.
\end{prop}
\begin{proof}
  If $n \notin \Gamma (d/g)$, then $\ps_a$ does not vanish on $\V_d$
  by Corollary \ref{LL_cor}, thus $(n,d,a)$ is good. The second
  assertion follows from the fact that
  $\Gamma (d/g) \subseteq \Gamma (d)$ for any divisor $g$ of $d$.
\end{proof}
 
\begin{remark}
  The proof of Proposition \ref{n_notin_Gamma} uses a power sum as the
  symmetric polynomial of degree $a$.  It seems that we might be able
  to use Theorem \ref{LL} to handle more cases by using some other
  symmetric polynomial $f$.  While it is possible that
  $n \in \Gamma(d)$ and $f \in R^\SymGp$ is homogeneous having $m$
  terms with $m \notin \Gamma(d)$, this only happens in two cases.

  The first case is $f = e_n$, the $n^{\rm th}$ elementary symmetric
  polynomial, which consists of a single term and does not vanish on
  $\V_d$.  In particular, this shows that if $n$ divides $a$, then
  $(n,d,a)$ is \good.

  The second case is essentially when $d$ is a power of a prime.  See
  Corollaries~\ref{prime power} and \ref{essentially} below.  In
  fact, suppose two distinct primes $p,q$ divide $d$, $n \geq p+q$,
  $n \in \Gamma(d)$ and let $f$ be a non-constant symmetric polynomial
  having $m$ terms.  Then $n \geq p+q$ implies that
  $\binom{n}{2} \geq (p-1)(q-1)$.  Thus, if
  $m \geq \binom{n}{2}$, then $m \geq (p-1)(q-1)$, which implies
  $m \in \Gamma(pq)$ (cf. \cite[Thm.~2.1.1]{Sylvester}). Since
  $\Gamma (pq) \subseteq \Gamma (d)$, we deduce that
  $m \geq \binom{n}{2}$ implies $m\in \Gamma (d)$.  Therefore, if
  $m \notin \Gamma(d)$, then $m < \binom{n}{2}$. Since we are
  assuming $n\in \Gamma(d)$, this implies that $f = \lambda e_n$ for
  some scalar $\lambda$.
\end{remark}

\begin{prop}\label{S semigroup}
  Define $S := \left\{ q : q \mid d, n \notin \Gamma(d/q)
  \right\}$. If $a$ lies in the numerical semi-group generated by $S$,
  then the triple $(n,d,a)$ is \good.
\end{prop}
  
\begin{proof}
  By the hypothesis, we can write $a = \sum_{i=1}^r \lambda_i$, where
  $\lambda_i \in S$ for $i=1,2,\dots,r$ and
  $\lambda_1 \geq \lambda_2 \geq \dots \geq \lambda_r$. Then
  $\lambda = (\lambda_1,\lambda_2,\dots,\lambda_r)$ is a partition of
  $a$ and $\ps_\lambda$ is a symmetric polynomial of degree $a$.

  Since $\lambda_i \in S$, we have that $\lambda_i \mid d$, hence
  $\gcd (\lambda_i,d) = \lambda_i$. Moreover,
  $n\notin \Gamma (d/\lambda_i)$. Therefore Corollary \ref{LL_cor}
  implies that $\ps_{\lambda_i}$ does not vanish on $\V_d$. Since this
  holds for all indices $i=1,2,\dots,r$, we conclude that
  $\ps_\lambda (Q)$ does not vanish on $\V_d$. Therefore $(n,d,a)$ is
  \good.
\end{proof}

\begin{remark}
  Note that $d \in S$ always.  Furthermore, if
  $d = p_1^{b_1} p_2^{b_2}\cdots p_t^{b_t}$ is the prime factorization
  of $d$, then the set
  $$\left\{ \frac{d}{p_i^{b_i}} : p_i \nmid n\right\}$$
  is a subset of $S$.
\end{remark}

\begin{remark}
  Proposition \ref{S semigroup} remains true if we use $S\cup \{n\}$
  instead of $S$. In fact, if $a$ lies in the numerical semi-group
  generated by $S\cup {n}$, then $a = b + cn$, where $b,c$ are
  positive integers and $b$ lies in the numerical semi-group generated
  by $S$. By the proof of Proposition \ref{S semigroup}, there exists
  $\lambda \vdash b$ such that $\ps_\lambda$ does not vanish on
  $\V_d$. At the same time, the elementary symmetric polynomial $e_n$
  does not vanish on $\V_d$. Therefore $\ps_\lambda e_n^c$ is a
  homogeneous symmetric polynomial of degree $a$ which does not vanish
  on $\V_d$.
\end{remark}

\begin{prop}\label{coprime lemma}
  Suppose that $n \in \Gamma(d)$ and
  $a \notin \Gamma(d)$.  Then $(n,d,a)$ is \bad.
\end{prop}
    
\begin{proof}
  Since $n \in \Gamma(d)$, there exists $Q \in \V_d$ such that
  $\ps_1(Q)=0$ by Theorem \ref{LL}. If
  $\lambda=(\lambda_1,\lambda_2,\dots,\lambda_r) \vdash a$, then some
  part $\lambda_t$ is coprime to $d$ since $a \notin \Gamma(d)$.
  Hence, by Remark \ref{Galois remark}, we have $\ps_{\lambda_t}(Q)=0$
  and thus $\ps_\lambda(Q) = 0$. The reasoning holds for all
  $\lambda \vdash a$.  Therefore $(n,d,a)$ is \bad{} by Lemma
  \ref{power sums suffice}.
\end{proof}

\begin{prop}\label{two proofs}
  Let $g: = \gcd(d,n)$.  If $g\nmid a$, then $(n,d,a)$ is \bad.
\end{prop}

\begin{proof}
  Let $\omega$ be a primitive $g^{\rm th}$ root of unity and define
  $Q = (\omega,\omega^2,\dots,\omega^n) \in \V_d$. Observe that
  $\omega^i = \omega^{i+gj}$ for all $i,j\in\ZZ$. Hence, using the
  auxiliary variable $y$, we have
  $$\prod_{i=1}^n (y-\omega^i) = \left[ \prod_{i=1}^g (y-\omega^i)
  \right]^{n/g} = (y^g - 1)^{n/g}.$$
  On the other hand
  $$\prod_{i=1}^n (y-\omega^i) = \sum_{j=0}^n (-1)^j e_j (Q) y^{n-j}.$$
  By comparing the two expressions, we deduce that $e_j (Q) = 0$
  whenever $g\nmid j$. Thus the only symmetric polynomials potentially
  not vanishing at $Q$ are the ones in the subring
  $\CC [e_j : g\mid j]$. Note how the degree of any element in this
  subring is divisible $g$. Since $g\nmid a$, $(n,d,a)$ is \bad{} by
  Corollary \ref{single Q}.
\end{proof}

\begin{prop}\label{a_bigger_than}
  Let $g:=\gcd (d,n)$ and assume that
  $$a \geq \frac{(n-g)(d-g)}{g}.$$
  Then $(n,d,a)$ is \bad{} if and only if $g\nmid a$.
\end{prop}

\begin{proof}
  If $g\nmid a$, then the triple is \bad{} by Proposition \ref{two
    proofs}.

  Assume $g\mid a$ and let $a'=a/g$, $n'=n/g$, and $d'=d/g$. The
  inequality in the assumption gives
  $$a' = \frac{a}{g} \geq \frac{n-g}{g} \frac{d-g}{g}
  = (n'-1)(d'-1).$$
  By \cite[Thm.~2.1.1]{Sylvester}, $a'$ belongs to the numerical
  semi-group generated by $d'$ and $n'$. Thus we can write
  $a' = sd'+tn'$, for some non-negative integers $s$ and
  $t$. Multiplying by $g$, we obtain $a = sd+tn$. This equality
  implies that the homogeneous symmetric polynomial
  $f := \ps_d^s e_n^t$ has degree $a$. For all $Q\in \V_d$, we have
  $\ps_d (Q) = n \neq 0$. Moreover, $e_n$ does not vanish on
  $\V_d$. Therefore $f$ does not vanish on $\V_d$ and the triple
  $(n,d,a)$ is \good.
\end{proof}

\subsection{Triples and prime factors}
\label{sec:triples-with-prime}
Here we analyze the property of a triple $(n,d,a)$ being \good{} or
\bad{} in relation to certain prime factors of $n$, $d$, and $a$. We
begin by developing some technical results.

Let $z_1,z_2,\dots,z_n$ be $d^{\rm th}$ roots of unity and consider
the point $Q=(z_1,z_2,\dots,z_n) \in \AA^n$.  For an integer $v$, we
say that $Q$ is \emph{$v$-symmetric} if, given a primitive
$v^{\rm th}$ root of unity $\epsilon$, there exists $\tau \in \SymGp$
such that
$$(\epsilon z_1, \epsilon z_2, \dots, \epsilon z_n) =
(z_{\tau(1)},z_{\tau(2)},\dots,z_{\tau(n)}).$$ In other words, $Q$ is
$v$-symmetric if rotating each of the complex coordinates $z_i$ by
$2\pi/v$ radians produces a point in the $\SymGp$-orbit of $Q$.
Note that $v\mid d$ because
$1=z_{\tau (1)}^d=\epsilon^d z_1^d=\epsilon^d$ and $\epsilon$ is
primitive.
  
\begin{lemma}\label{is symmetric}
  The point $Q \in \V_d \subset \AA^n$ is $v$-symmetric if and only if
  $v \mid n$ and $e_j(Q)=0$ for all $j$ such that $v\nmid j$.
\end{lemma}
  
\begin{proof}
  First suppose that $Q$ is $v$-symmetric.  The coordinates of $Q$
  split into orbits under the cyclic group of order $v$ acting on the
  complex plane by rotation.  Since $Q \in \V_d$, we have $z_i \neq 0$
  for all $i$. Therefore all the above orbits have cardinality $v$ and
  $v\mid n$.

  Since $Q$ is $p$-symmetric, there is a primitive $v^{\rm th}$ root
  of unity $\epsilon$ such that, up to reordering, we may write
  $z_{jv+i}= \epsilon^i \omega_j$ for $1 \leq i \leq v$,
  $1 \leq j \leq n/v$, and for some $d^{\rm th}$ roots of unity
  $\omega_j$.  Using the auxiliary variable $y$, we have
  \begin{align*}
    \sum_{j=1}^n (-1)^{j} e_j(Q) y^j
    & = \prod_{i=1}^n(y-z_i) = \prod_{j=1}^{n/v} \prod_{i=1}^v (y-\epsilon^i \omega_j) = 
    \prod_{j=1}^{n/v} \omega_j^v \prod_{i=1}^v (y/\omega_j-\epsilon^i)\\
    &= \prod_{j=1}^{n/v} \omega_j^v \left( (y/\omega_j)^v - 1\right) = \prod_{j=1}^{n/v} (y^v - \omega_j^v).
  \end{align*}
  Thus $e_j(Q)=0$ whenever $v\nmid j$.
    
  Conversely, suppose that $v\mid n$, $Q\in\V_d$ and $e_j(Q)=0$
  whenever $j\nmid v$.  We have
  $$\prod_{i=1}^n(y-z_i) = \sum_{j=1}^n (-1)^{j} e_j(Q) y^j = f(y^v),$$
  where $f$ is a polynomial in one variable.  At the same time
  $$\prod_{i=1}^n(y-\epsilon z_i) = \epsilon^n \prod_{i=1}^n
  (y/\epsilon - z_i) = \epsilon^n f((y/\epsilon)^v) = f(y^v).$$
  Therefore, comparing factors, we deduce that $Q$ is symmetric.
\end{proof}
  
\begin{lemma}\label{p-symmetric}
  Suppose $Q=(z_1,z_2,\dots,z_n) \in \V_d$ is $v^m$-symmetric and
  $(z_1^{v^m},z_2^{v^m},\dots,z_n^{v^m})$ is $v$-symmetric.  Then $Q$
  is $v^{m+1}$-symmetric.
\end{lemma} 
  
\begin{proof}
  Proceeding as in the proof of Lemma \ref{is symmetric}, $Q$ being
  $v^m$-symmetric implies the existence of a primitive
  $(v^m)^{\rm th}$ root of unity $\epsilon$ such that, up to
  reordering, we may write $z_{jv^m+i}= \epsilon^i \omega_j$ for
  $1 \leq i \leq v^m$, $1 \leq j \leq n/v^m$, and for some $d^{\rm th}$
  roots of unity $\omega_j$.  Using the auxiliary variable $y$, we
  have
  \begin{equation}\label{eq:5}
    \prod_{i=1}^n (y-z_i^{v^m}) =
    \prod_{j=1}^{n/v^m} \prod_{i=1}^{v^m} (y-\omega_j^{v^m})=
    \prod_{j=1}^{n/v^m} (y-\omega_j^{v^m})^{v^m}=
    \left( \prod_{j=1}^{n/v^m} (y-\omega_j^{v^m}) \right)^{v^m}.
  \end{equation}
  Since $(z_1^{v^m},z_2^{v^m},\dots,z_n^{v^m})$ is $v$-symmetric,
  Lemma \ref{is symmetric} implies
  \begin{equation}
    \label{eq:6}
    \prod_{i=1}^n (y-z_i^{v^m}) = f(y^v),
  \end{equation}
  for some polynomial $f$ in one variable. To only way to reconcile
  equations \eqref{eq:5} and \eqref{eq:6} is if
  \begin{equation*}
    \label{eq:7}
    \prod_{j=1}^{n/v^m} (y-\omega_j^{v^m}) = g(y^v),
  \end{equation*}
  for some polynomial $g$ in one variable.
  Therefore we must have
  \begin{equation*}
    \begin{split}
      \prod_{i=1}^n (y-z_i)
      &=\prod_{j=1}^{n/v^m} \prod_{i=1}^{v^m} (y-z_{jv^m+i})
      =\prod_{j=1}^{n/v^m} \prod_{i=1}^{v^m} (y-\epsilon^i \omega_j)=\\
      &=\prod_{j=1}^{n/v^m} (y^{v^m}- \omega_j^{v^m}) =
      g \left( (y^{v^m})^v \right) = g (y^{v^{m+1}}).
    \end{split}
  \end{equation*}
  Using Lemma \ref{is symmetric} again, we conclude that $Q$ is
  $v^{m+1}$-symmetric.
\end{proof}

\begin{prop}\label{symmetric case}
  Let $p$ be prime and suppose that all points
  $Q\in \V_d \subseteq \AA^n$ with $\ps_1 (Q)=0$ are $p$-symmetric.
  Let $g:=\gcd (d,n)$ and assume $p\mid g$.  Then $(n,d,a)$ is \bad{}
  if and only if $g\nmid a$.
\end{prop}

\begin{proof}
  If $g\nmid a$, then $(n,d,a)$ is \bad{} by Proposition \ref{two
    proofs}.

  We prove the other implication by contradiction, so suppose that
  $g\mid a$.  Let $n=p^r n'$, $d=p^s d'$ and $a = p^t a'$, where
  $\gcd(p,n')=\gcd(p,d')=\gcd(p,a')=1$.  Set $k = \min\{r,s\}$. Since
  $p^k \mid g$, the condition $g\mid a$ implies $p^k \mid a$ and
  therefore $k\leq t$.

  The hypothesis $p\mid g$ implies $s\geq 1$; hence $p\in \Gamma(d)$.
  At the same time, $p\mid g$ also implies $r\geq 1$; hence
  $n\in \Gamma(d)$.  Thus, by Theorem \ref{LL}, there exists
  $Q\in \V_d \subseteq \AA^n$ such that $\ps_1 (Q) =0$.  By the
  hypothesis, $Q$ is $p$-symmetric. However, $Q$ is not
  $p^{k+1}$-symmetric because either $p^{k+1} \nmid n$ or
  $p^{k+1} \nmid d$. Therefore there is an integer $m$, with
  $1\leq m\leq k$, such that $Q$ is $p^m$-symmetric but not
  $p^{m+1}$-symmetric.

  Now suppose that $\ps_{p^m} (Q) = 0$. Then we would have
  $$\ps_1 (z_1^{p^m},z_2^{p^m},\dots,z_n^{p^m}) = \ps_{p^m} (Q) = 0.$$
  Our hypothesis would imply that
  $(z_1^{p^m},z_2^{p^m},\dots,z_n^{p^m})$ is $p$-symmetric. However,
  Lemma \ref{p-symmetric} would give that $Q$ is $p^{m+1}$-symmetric,
  contradicting our choice of $m$. Therefore $\ps_{p^m} (Q) \neq 0$.
  Thus the homogeneous polynomial $(\ps_{p^m})^{a'p^{t-m}}$ has degree
  $a$ and does not vanish at $Q$.  We conclude that $(n,d,a)$ is
  \good{} by Corollary \ref{single Q}.
\end{proof}

In \cite{vanishing sums}, Lam and Leung consider sequences
$z_1,z_2,\dots,z_n$ with each $z_i$ a $d^{\rm th}$ root of unity and
whose sum is 0, in particular, points
$Q = (z_1,z_2,\dots,z_n) \in \V_d$ such that $\ps_1 (Q) =0$.
Corollary~3.4 of \cite{vanishing sums} shows that if $d=p^r$ is a
prime power, then $Q$ must be $p$-symmetric.  This yields the
following corollary of Proposition~\ref{symmetric case}.
\begin{cor}
  \label{prime power}
  Suppose $d=p^s$ for some prime $p$ and positive integer $s$.  Let
  $g:=\gcd(d,n)$. Then $(n,d,a)$ is \bad{} if and only if $g\nmid a$.
\end{cor}
\begin{proof}
  If $p\mid g$, then the result follows from \cite[Cor.~3.4]{vanishing
    sums} and Proposition~\ref{symmetric case}.

  Assume $p\nmid g$. In this case, $g=1\mid a$ so we must show that
  $(n,d,a)$ is \good.  Note that $p\nmid g$ implies $p\nmid n$. Hence
  $n\notin \Gamma (d) = \langle p \rangle$.  Therefore $(n,d,a)$ is
  \good{} by Proposition \ref{n_notin_Gamma}.
\end{proof}
      
Lam and Leung also showed that if $(z_1,z_2,\dots,z_n)$ is not
$p$-symmetric for all primes $p$ dividing $d$, then
$n \geq p_1(p_2-1) + p_3 - p_2$, where $p_1 < p_2 <p_3$ are the three
smallest primes dividing $d$ \cite[Thm.~4.8]{vanishing sums}.  This
yields the following corollary of Proposition~\ref{symmetric case}.

\begin{cor}\label{essentially}
  Suppose that at least two distinct primes divide $d$ and that
  $n < p + q$ where $p$ and $q$ are the smallest two distinct primes
  dividing $d$. Let $g:=\gcd(d,n)$.  Then $(n,d,a)$ is \bad{} if and
  only if $g\nmid a$.
\end{cor}

\begin{proof}
  Let $d = p^s \prod_{i=1}^m q_i^{s_i}$ be the prime factorization of
  $d$, where $p < q_1 < q_2 < \dots < q_m$.  Suppose that
  $Q = (z_1,z_2,\dots,z_n) \in \V_d$ is such that $\ps_1(Q)=0$.  Since
  $$p(q_1-1) + q_2-q_1 \geq 2(q_1-1) + q_2 - q_1 = (q_2-1) + q_1 > p + q_1 > n,$$
  \cite[Thm.~4.8]{vanishing sums} implies that every non-empty
  minimal subset $I \subset \{1,2,\dots,n\}$ such that
  $\sum_{i \in I} z_i = 0$ corresponds to a $v$-symmetric point
  $(z_i : i \in I)$, where $v$ is a prime dividing $d$.  Moreover, $v$
  divides the cardinality of $I$.
  Clearly, we may partition $\{1,2,\dots,n\}$ into a disjoint union
  $I_1 \sqcup I_2 \sqcup \dots \sqcup I_t$ of such minimal subsets.
  Thus $n = \# I_1 + \# I_2 + \dots + \# I_t$.  Since the cardinality
  of each $I_j$ is either $p$ or some $q_i$, the hypothesis $n<p+q_1$
  implies we must have either $t=1$ and $n= \# I_1 = q_i$ for some
  $i$, or else $\# I_j =p$ for all $j$ and $n=tp$.

  Thus there are two possibilities: either $n=q_i$ for some $q_i$, or
  else $n=pt$.  In the former case, $q_i \mid g$ and every
  $Q \in \V_d$ with $\ps_1(Q)=0$ is $q_i$-symmetric.  In the latter
  case, $p\mid g$ and every $Q \in \V_d$ with $\ps_1(Q)=0$ is
  $p$-symmetric.  Thus the hypotheses of Proposition~\ref{symmetric
    case} are satisfied (either with the prime $q_i$ or with $p$).
\end{proof}

\subsection{Generating good and bad triples}
\label{sec:generating-triples}  

We illustrate how to obtain more \good{} and \bad{} triples from the
ones already at our disposal.

\begin{prop}\label{inc_bad}
  Let $k$ be a positive integer.
  \begin{enumerate}
  \item\label{inc1} If $(n,d,a)$ is \bad, then $(n,kd,a)$ is also
    \bad.
  \item\label{inc2} If $(n,d,a)$ is \bad, then $(kn,d,a)$ is also
    \bad.
  \item\label{inc5} If $(n,d,a)$ is \bad, then $(kn,kd,ka)$ is also
    \bad.
  \end{enumerate}  
\end{prop}

\begin{proof}
  Suppose that $(n,d,a)$ is \bad.  By Corollary~\ref{single Q}, there
  is a point $Q=(z_1,z_2,\dots,z_n) \in \V_{d} \subset \AA^{n}$
  such that $f(Q)=0$ for all $f \in R_{a}^\SymGp$.  Assertion
  (\ref{inc1}) follows immediately since $\V_d \subset \V_{kd}$.
 
  For the second assertion, choose a point
  $Q = (z_1,z_2,\dots,z_n)\in \V_d$. Define the point
  $Q' = (z'_1,z'_2,\dots,z'_{kn}) \in \V_{kd} \subset \AA^{kn}$ by
  $z'_{i+n(j-1)} := z_i$ for $1 \leq i \leq n$ and $1 \leq j \leq k$.
  Assume, by way of contradiction, that there exists
  $f' \in \CC[x_1,x_2,\ldots,x_{kn}]_{a}^ {{\mathfrak{S}_{kn}}}$ such
  that $f'(Q') \neq 0$. The polynomials $\ps_\lambda$ with $\lambda$ a
  partition of $a$ whose parts do not exceed $kn$ form a basis of
  $\CC[x_1,x_2,\ldots,x_{kn}]_{a}^ {{\mathfrak{S}_{kn}}}$. Then
  $f'(Q')\neq 0$ implies that there exists a partition
  $\lambda=(\lambda_1,\lambda_2,\dots,\lambda_r) \vdash a$ with
  $\ps_{\lambda}(Q') \neq 0$.  Hence $\ps_{\lambda_t}(Q') \neq 0$ for
  all $t=1,2,\dots,r$.  Since
  $$\ps_{\lambda_t}(Q') = k z_1^{\lambda_t} + k z_2^{\lambda_t} + \dots + 
  k z_n^{\lambda_t} = k \ps_{\lambda_t}(Q),$$ we have
  $\ps_{\lambda_t}(Q) \neq 0$ for all $t=1,2,\dots,r$, and therefore
  $\ps_\lambda (Q) \neq 0$.  Because $Q\in \V_d$ is arbitrary, Lemma
  \ref{power sums suffice} shows $(n,d,a)$ is not \bad.  This
  contradicts the assumption, thus proving (\ref{inc2}).

  Now we prove part (\ref{inc5}). By contradiction, assume
  $(kn,kd,ka)$ is not \bad.  Given $Q\in \V_d$, we will construct
  $f\in R^\SymGp_a$ such that $f(Q)\neq 0$, which will prove $(n,d,a)$
  is not \bad.  Consider the primitive $d^{\rm th}$ root of unity
  $\zeta := e^{2\pi i/d}$. We have
  $Q = (\zeta^{b_1}, \zeta^{b_2}, \dots, \zeta^{b_n})$ for some
  positive integers $b_1,b_2,\dots,b_n$. Let
  $\omega := e^{2\pi i/(kd)}$; observe that $\omega$ is a
  $(kd)^{\rm th}$ root of unity and $\omega^k = \zeta$. Define the
  point $Q' = (z'_1,z'_2,\dots,z'_{kn}) \in \V_{kd} \subset\AA^{kn}$
  by $z'_{k(j-1)+i} := \omega^{b_j + id}$ for $1\leq i\leq k$ and
  $1\leq j\leq n$. Since we have assumed that $(kn,kd,ka)$ is not
  \bad, by Lemma \ref{power sums suffice}, there exists a partition
  $\lambda = (\lambda_1,\lambda_2,\dots,\lambda_r) \vdash ka$ such
  that $\ps_\lambda (Q') \neq 0$. In particular,
  $\ps_{\lambda_t} (Q') \neq 0$ for all $t=1,2,\dots,r$.

  Using the auxiliary variable $y$, we can write
  \begin{align*}
    \prod_{t=1}^{kn} (y-z'_t) 
    &=\prod_{j=1}^{n}\prod_{i=1}^k (y-\omega^{b_j + id}) =
      \prod_{j=1}^{n}\prod_{i=1}^k \omega^{b_j}(y/\omega^{b_j}-\omega^{id}) =\\
    &=\prod_{j=1}^{n} \omega^{kb_j}
      \prod_{i=1}^k(y/\omega^{b_j}-(\omega^d)^i).
  \end{align*}
  Since $\omega^d$ is a primitive $k^{\rm th}$ root of unity, the $k$
  elements $(\omega^d)^1,(\omega^d)^2,\dots,(\omega^d)^k$ are all the
  $k^{\rm th}$ roots of unity. Therefore we get
  $$\prod_{i=1}^k (y/\omega^{b_j}-(\omega^d)^i) = (y/\omega^{b_j})^k -1.$$
  Combining the two previous equations, we obtain
  $$\prod_{t=1}^{kn} (y-z'_t) = \prod_{j=1}^{n} \omega^{k b_j}
  [(y/\omega^{b_j})^k -  1] =  \prod_{j=1}^{n} (y^k-\zeta^{b_j}).$$
  On the other hand, we have
  $$\prod_{t=1}^{kn} (y-z'_t) = \sum_{j=0}^{kn} (-1)^j e_j (Q') y^{kn-j}.$$
  By comparing these expressions, we deduce that $e_j (Q') = 0$
  whenever $k \nmid j$. This implies that every homogeneous polynomial
  in $\CC [x_1,x_2,\dots,x_{kn}]^{\mathfrak{S}_{kn}}$ whose degree is
  not divisible by $k$ vanishes at $Q'$.

  Thus the above integers $\lambda_1,\lambda_2,\dots,\lambda_r$ are
  all divisible by $k$ and we set $c_t := \lambda_t / k$ for all
  $i=1,2,\dots,r$.  We have
  \begin{align*}
    \ps_{\lambda_t}(Q')
    &= \ps_{kc_t}(Q') = \sum_{s=1}^{kn} (z'_s)^{kc_t} =
      \sum_{j=1}^{n} \sum_{i=1}^k (\omega^{b_j+id})^{kc_t} =
      \sum_{i=1}^{k} \sum_{j=1}^{n}  (\omega^k)^{(b_j+id)c_t} =\\
    & = \sum_{i=1}^{k} \sum_{j=1}^{n}  \zeta^{(b_j+id)c_t} =
      \sum_{i=1}^k \sum_{j=1}^{n} (\zeta^{b_j})^{c_t} =
      \sum_{i=1}^k \ps_{c_t}(Q) = k \ps_{c_t}(Q).
  \end{align*}
  We deduce that $\ps_{c_t}(Q) \neq 0$ for all $i=1,2,\dots,r$.
  Define $f := \prod_{t=1}^r \ps_{\lambda_t / k} \in R$ and observe
  that $f$ is an element of $R^\SymGp_a$ with $f(Q)\neq 0$. This
  concludes the proof.
\end{proof}

As the following example illustrates, $(n,d,a)$ being \bad{} does not
imply that any of $(n,d,ka)$, $(n,kd,ka)$, or $(kn,d,ka)$ is \bad.
      
\begin{example}
  Consider $(n,d,a)=(8,15,4)$. Since
  $8 = 5+3 \in \Gamma(d) = \Gamma(15)$ and $4 \notin \Gamma(d)$, we
  see that $(8,15,4)$ is \bad{} by Proposition~\ref{coprime lemma}.

  Let $k=2$. The triples $(n,d,ka)=(8,15,8)$ and $(n,kd,ka)=(8,30,8)$
  are \good{} because $e_8$ clearly does not vanish on $\V_{15}$ nor
  on $\V_{30}$.

  Now consider the triple $(kn,d,ka)=(16,15,8)$. Observe that
  $$S = \{q : q\mid 15, 16\notin \Gamma(15/q)\} = \{3,5,15\},$$
  hence the numerical semi-group $\langle 3,5\rangle$ generated by $S$
  contains $ka=8$. Therefore $(16,15,8)$ is \good{} by Proposition
  \ref{S semigroup}.
\end{example}

\begin{remark}
  Consider the triple $(n,d,a)$ and let $g:=\gcd(n,d)$.  By
  Proposition~\ref{two proofs}, $(n,d,a)$ is \bad{} if $g\nmid a$.
  Thus we suppose that $g\mid a$.  By Proposition \ref{inc_bad}
  (\ref{inc5}), if $(n,d,a)$ is \good{}, then $(n/g,d/g,a/g)$ is also
  \good.
\end{remark}

\begin{prop}\label{inc_good}
  Let $k$ be a positive integer.
  \begin{enumerate}
  \item\label{inc3} If $(n,d,a)$ is \good, then $(n,d,ka)$ is also
    \good.
  \item\label{inc4} If $(n,d,a)$ is \good, then $(n,kd,ka)$ is also
    \good.
  \end{enumerate}
\end{prop}

\begin{proof}
  Suppose that $(n,d,a)$ is \good.  This implies that there exists
  $f \in R_{a}^\SymGp$ which does not vanish on $\V_d$.  Assertion
  (\ref{inc3}) now follows since $f^k \in R_{ka}^\SymGp$ also does not
  vanish on $\V_d$.
    
  To prove (\ref{inc4}), define
  $$f'(x_1,x_2,\dots,x_n) := f(x_1^k,x_2^k,\dots,x_n^k) \in
  R_{ka}^\SymGp.$$ For every point
  $Q'=(z_1,z_2,\dots,z_n) \in \V_{kd}$, the point
  $Q = (z_1^k,z_2^k,\dots,z_n^k)$ lies in $\V_d$; moreover,
  $f'(Q') =f(Q) \neq 0$.  Thus $(n,kd,ka)$ is \good.
\end{proof}

As the following example illustrates, $(n,d,a)$ being \good{} does not
imply that any of $(kn,d,a)$, $(n,kd,a)$, $(kn,kd,a)$, $(kn,d,ka)$ or
$(kn,kd,ka)$ is \good.

\begin{example}
  Consider $(n,d,a)=(4,15,1)$.  Since
  $n=4 \notin \Gamma(d) = \Gamma(15) = \langle 3,5 \rangle$, we see
  that $(4,15,1)$ is \good{} by Proposition \ref{S semigroup}.

  Now consider $k=2$ and $(kn,d,a)=(8,15,1)$.  Since
  $8 \in \Gamma(15)$ and $1\notin \Gamma(15)$, we see that $(8,15,1)$
  is \bad{} by Proposition \ref{coprime lemma}.  The triple
  $(n,kd,a)=(4,30,1)$ is \bad{} for similar reasons.  Then
  $(kn,kd,a)=(8,30,1)$ is \bad{} as well by
  Proposition~\ref{inc_bad}~(\ref{inc2}).

  We claim the triple $(kn,d,ka)=(8,15,2)$ is also \bad{}.  Using the
  fact that $(8,15,1)$ is bad, we deduce that there exists
  $Q\in \V_{15}$ such that $\ps_1 (Q) = 0$.  Since $2$ and $15$ are
  coprime, Remark \ref{Galois remark} implies $\ps_2 (Q) = 0$. Given
  that $\ps_2$ and $\ps_{(1,1)} = \ps_1^2$ form a basis of the
  symmetric polynomials of degree 2, their simultaneous vanishing at
  $Q$ implies the claim by Lemma \ref{power sums suffice}.  Finally,
  the claim just proved, together with
  Proposition~\ref{inc_bad}~(\ref{inc1}), implies that
  $(kn,kd,ka)=(8,30,2)$ is \bad.
\end{example}

\section{Regular Sequences of Type $S^{(2,2)} \oplus S^{(4)} \oplus S^{(4)}$}

Throughout this section we fix $n=4$, so $R=\CC[x_1,x_2,x_3,x_4]$.

As proved in \cite[Prop.~2.5]{SCI1}, there exist homogeneous regular
sequences $g_1,g_2,f_1,f_2$ in $R$ such that $g_1,g_2$ form a basis of
a graded representation isomorphic to $S^{(2,2)}$ and $f_1,f_2$ are
symmetric polynomials. If $I \subset R$ is the ideal generated by
$g_1,g_2,f_1,f_2$, then $I/\mathfrak{m} I$ is isomorphic to
$S^{(2,2)} \oplus S^{(4)} \oplus S^{(4)}$. Setting
$a := \deg(g_1) = \deg(g_2)$, $c:= \deg(f_1)$ and $d := \deg(f_2)$, we
seek the possible tuples $(a,a,c,d)$ corresponding to regular
sequences $g_1,g_2,f_1,f_2$ of type
$S^{(2,2)} \oplus S^{(4)} \oplus S^{(4)}$.

\subsection{Sequences in low degree}
\label{sec:sequences-low-degree}

We recall some facts of invariant theory; more details can be found in
\cite[Ch.~3,4]{refl_groups}. There is an isomorphism
$R \cong R^{\mathfrak{S}_4} \otimes_\CC R/(e_1,e_2,e_3,e_4)$ of graded
$\mathfrak{S}_4$-representations. The symmetric group acts trivially
on $R^{\mathfrak{S}_4}$. On the other hand, the coinvariant algebra
$R/(e_1,e_2,e_3,e_4)$ is isomorphic to the regular representation of
$\mathfrak{S}_4$. We worked out the graded character of
$R/(e_1,e_2,e_3,e_4)$ in \cite[Ex.~3.1]{SCI1}. In particular,
$R/(e_1,e_2,e_3,e_4)$ contains two copies of the irreducible
representation $S^{(2,2)}$, one in degree 2 and one in degree 4.

Let us find an explicit description of these two
representations. Specht's original construction shows that the
polynomials
\begin{equation}\label{eq:1}
  (x_1-x_2)(x_3-x_4), (x_1-x_3)(x_2-x_4)
\end{equation}
span a copy of $S^{(2,2)}$ inside the degree 2 component of $R$
(cf. \cite[\S 7.4, Ex.~17]{Fulton}).  Now observe that the polynomials
\begin{equation}\label{eq:2}
  (x_1^2-x_2^2)(x_3^2-x_4^2), (x_1^2-x_3^2)(x_2^2-x_4^2)
\end{equation}
behave in the same way under the action of $\mathfrak{S}_4$. Therefore
they span a copy of $S^{(2,2)}$ inside the degree 4 component of
$R$. Note also that the polynomials in \eqref{eq:1} and \eqref{eq:2}
do not belong to the ideal $(e_1,e_2,e_3,e_4)$. Therefore their
residue classes span the desired copies of $S^{(2,2)}$ inside
$R/(e_1,e_2,e_3,e_4)$.

Using the isomorphism
$R \cong R^{\mathfrak{S}_4} \otimes_\CC R/(e_1,e_2,e_3,e_4)$ together
with our construction above, we can establish the following
fundamental fact: any copy of $S^{(2,2)}$ contained inside the degree
$a$ component of $R$ is spanned by
\begin{equation}\label{eq:3}
  \begin{split}
    g_1 &= h (x_1-x_2)(x_3-x_4) + h' (x_1^2-x_2^2)(x_3^2-x_4^2),\\
    g_2 &= h (x_1-x_3)(x_2-x_4) + h' (x_1^2-x_3^2)(x_2^2-x_4^2)
  \end{split}
\end{equation}
for some symmetric polynomials $h$ of degree $a-2$ and $h'$ of degree
$a-4$.

Thus, when searching for degree tuples $(a,a,c,d)$ corresponding to
regular sequences $g_1,g_2,f_1,f_2$ of type
$S^{(2,2)} \oplus S^{(4)} \oplus S^{(4)}$, we can assume that
$g_1,g_2$ have the form given in equation \eqref{eq:3}.

We consider the cases where $a \leq 4$ first.
Clearly we must have $a \geq 2$.

\begin{prop}
  Let $a=2$ or $4$. A regular sequence of type
  $S^{(2,2)} \oplus S^{(4)} \oplus S^{(4)}$ with degree tuple
  $(a,a,c,d)$ exists if and only if $d\geq 2$. If $a=3$, then no such
  sequence exists.
\end{prop}
\begin{proof}
  Let $a=2$. We form polynomials $g_1,g_2$ as in equation
  \eqref{eq:3}. By degree considerations, $h$ is a unit and
  $h'=0$. Therefore we may take
  \begin{equation*}
    g_1=(x_1-x_2)(x_3-x_4), g_2=(x_1-x_3)(x_2-x_4).
  \end{equation*}
  Now we need symmetric polynomials $f_1,f_2$ such that
  $g_1,g_2,f_1,f_2$ is a regular sequence. Note that $f_1,f_2$ cannot
  both be linear, otherwise they would be scalar multiples of
  $e_1$. However, if we assume that $d = \deg (f_2) \geq 2$, then we
  can write $d=2p+3q$, where $p,q$ are non-negative integers, and set
  $f_1:=e_1^c$, $f_2:=e_2^p e_3^q$. The sequence $g_1,g_2,f_1,f_2$ is
  regular with degree tuple $(2,2,c,d)$.

  Now let $a=4$. We need $h'$ to be a unit. In fact, we can take
  $h'=1$ and $h=e_2$; this gives
  \begin{equation*}
    \begin{split}
      &g_1=e_2 (x_1-x_2)(x_3-x_4) + (x_1^2-x_2^2)(x_3^2-x_4^2),\\
      &g_2=e_2 (x_1-x_3)(x_2-x_4) + (x_1^2-x_3^2)(x_2^2-x_4^2).
    \end{split}
  \end{equation*}
  Again $f_1,f_2$ cannot both be linear. In fact, choosing the same
  $f_1,f_2$ as before gives a regular sequence $g_1,g_2,f_1,f_2$ with
  degree tuple $(4,4,c,d)$ for $d\geq 2$.

  Finally let $a=3$. In this case, $h'=0$ while $h$ is a scalar
  multiple of $e_1$. Thus $g_1,g_2$ have a common factor and do not
  form a regular sequence.
\end{proof}

\subsection{Sequences with $a\geq 5$}
\label{sec:sequences-with-ageq5}

Here we obtain general results about regular sequences
$g_1,g_2,f_1,f_2$ of type $S^{(2,2)} \oplus S^{(4)} \oplus S^{(4)}$
with degree tuple $(a,a,c,d)$ and $a \geq 5$.  We still refer to the
form of $g_1,g_2$ given in equation \eqref{eq:3}.

\begin{lemma}\label{reg seq conditions}
  Let
  \begin{equation*}
    \begin{split}
      &h_1 := h + h' (x_1 + x_2)(x_3 + x_4),\\
      &h_2 := h + h' (x_1 + x_3)(x_2 + x_4),
    \end{split}
  \end{equation*}
  so that $g_1 = (x_1-x_2)(x_3-x_4)h_1$ and
  $g_2 = (x_1-x_3)(x_2-x_4)h_2$.  The sequence $g_1,g_2,f_1,f_2$ is
  regular if and only if the sequences
  \begin{enumerate}
  \item $h,h',f_1,f_2$
  \item $(x_1-x_2)(x_3-x_4),(x_1-x_3)(x_2-x_4),f_1,f_2$
  \item $(x_1-x_2)(x_3-x_4),h_2,f_1,f_2$
  \end{enumerate}
  are regular.
\end{lemma}
\begin{proof}
  By Lemma \ref{lem:reg_seq_factors}, $g_1,g_2,f_1,f_2$ is regular if
  and only if
  \begin{enumerate}
    \renewcommand\labelenumi{(\roman{enumi})}
  \item $h_1,h_2,f_1,f_2$
  \item $(x_1-x_2)(x_3-x_4),(x_1-x_3)(x_2-x_4),f_1,f_2$
  \item $(x_1-x_2)(x_3-x_4),h_2,f_1,f_2$
  \item $h_1,(x_1-x_3)(x_2-x_4),f_1,f_2$
  \end{enumerate}
  are regular. Note that (ii) and (iii) are the same as (2) and (3)
  above. Moreover, the transposition $(2\,3)\in\mathfrak{S}_4$
  permutes (iii) and (iv), therefore it is enough to assume (iii) is
  regular. Thus the statement of the lemma will follow if we can prove
  that (1), (2), and (3) are regular if and only if (i), (ii), and
  (iii) are regular.

  Let us show that if (i) is regular then (1) is regular.  Since we
  have an equality of ideals
  $(h_1,h_2,f_1,f_2) = (h_1,h_2-h_1,f_1,f_2)$ and (i) is regular,
  $h_1,h_2-h_1,f_1,f_2$ is also regular.  Notice that
  \begin{equation}
    \label{eq:4}
    h_2 - h_1 = h' (x_1-x_4)(x_2-x_3).
  \end{equation}
  This implies that $h_1,h',f_1,f_2$ is regular. We deduce that (1) is
  regular, because of the equality
  $(h_1,h',f_1,f_2)=(h,h',f_1,f_2)$.

  Now assume that (1) and (3) are regular and let us prove that (i) is
  regular. Since (1) is regular, the equality
  $(h,h',f_1,f_2)=(h_1,h',f_1,f_2)$ implies that $h_1,h',f_1,f_2$ is
  regular. As previously observed, (3) being regular implies (iii) and
  (iv) are regular. Note that $(3\,4) h_1 = h_1$. Therefore, applying
  $(3\,4)$ to (iv), we obtain the regular sequence
  $h_1,(x_1-x_4)(x_2-x_3),f_1,f_2$. Since both $h_1,h',f_1,f_2$ and
  $h_1,(x_1-x_4)(x_2-x_3),f_1,f_2$ are regular, we can multiply their
  second elements to obtain a new regular sequence. By equation
  \eqref{eq:4}, this sequence is simply $h_1,h_2-h_1,f_1,f_2$. Finally
  the ideal equality $(h_1,h_2-h_1,f_1,f_2) = (h_1,h_2,f_1,f_2)$
  allows us to conclude that (i) is regular.
\end{proof}

By Lemma \ref{reg seq conditions}, a necessary condition for
$g_1,g_2,f_1,f_2$ to be a homogeneous regular sequence of type
$S^{(2,2)} \oplus S^{(4)} \oplus S^{(4)}$ and degrees $(a,a,c,d)$ is
that $(a-2,a-4,c,d)$ is a regular degree sequence.  In fact, we will
show this condition is also sufficient when $a \geq 5$.

\begin{prop}
  \label{pro:a_geq_5}
  Let $a \geq 5$.  Suppose that $(a-2,a-4,c,d)$ is a regular degree
  sequence for $n=4$.  Then there exists a homogeneous regular
  sequence of type $S^{(2,2)} \oplus S^{(4)} \oplus S^{(4)}$ and
  degrees $(a,a,c,d)$.
\end{prop}

\begin{proof}
  First we suppose that $a$ is even.  By Proposition \ref{beta
    condition}, $4! \mid (a-2)(a-4)cd$. Note that both $a-2$ and $a-4$
  are even, and at least one of them is divisible by 4. Therefore it
  is enough to account for 3 dividing $(a-2)(a-4)cd$. Moreover, we can
  assume $c\geq 2$ by condition $(\dag)$; in particular, we can write
  $c=2p+3q$ for some positive integers $p,q$. Table \ref{tab:even}
  contains our choices of polynomials $h,h',f_1,f_2$; one can easily
  verify that, in each case, $h,h',f_1,f_2$ is a regular sequence (see
  Remark \ref{rem:reg_seq_proof}).  The polynomials $g_1,g_2$ are
  obtained using equation \eqref{eq:3}. Using Lemma~\ref{reg seq
    conditions}, we conclude that, in each case, $g_1,g_2,f_1,f_2$ is
  a regular sequence of type
  $S^{(2,2)} \oplus S^{(4)} \oplus S^{(4)}$.

\begin{table}[htb]
  \centering
  $\begin{array}{|l|l|l|l|l|l|l|l|}
     \hline
     \multicolumn{4}{|c|}{\rm \bf Degrees} & \multicolumn{4}{|c|}{\rm\bf Symmetric\ Polynomials}\\ \hline
     \multicolumn{1}{|c|}{a-2} & \multicolumn{1}{|c|}{a-4} & \multicolumn{1}{|c|}{c} & \multicolumn{1}{|c|}{d}
                                                                                     & \multicolumn{1}{|c|}{h} & \multicolumn{1}{|c|}{h'} & \multicolumn{1}{|c|}{f_1} & \multicolumn{1}{|c|}{f_2}\\ \hline
     4\alpha+2 & 4\alpha & 3\gamma & d & e_2^{2\alpha+1} & e_4^\alpha & e_3^\gamma & e_1^d\\
     4\alpha & 4\alpha-2 & 3\gamma & d & e_4^\alpha & e_2^{2\alpha-1} & e_3^\gamma & e_1^d\\
     12\alpha+2 & 12\alpha & 2p+3q  & d & (e_2^3+e_3^2)e_4^{3\alpha-1} & (e_2^6+e_3^4+e_4^3)^\alpha & e_2^p e_3^q & e_1^d\\
     12\alpha & 12\alpha-2 & 2p+3q  & d & (e_2^6+e_3^4+e_4^3)^\alpha & (e_2^3+e_3^2)e_4^{3\alpha-2} & e_2^p e_3^q & e_1^d\\
     6\alpha,\ (2\nmid \alpha) & 6\alpha-2 & 2p+3q  & d & (e_2^3+e_3^2)^\alpha & e_4^{(3\alpha-1)/2} & e_2^p e_3^q & e_1^d\\
     6\alpha+2,\ (2\nmid \alpha) & 6\alpha & 2p+3q  & d & e_4^{(3\alpha+1)/2} & (e_2^3+e_3^2)^\alpha & e_2^p e_3^q & e_1^d\\
     \hline
   \end{array}$
   \bigskip
   \caption{$a$ even}\label{tab:even}
 \end{table}

 Next suppose that $a$ is odd. We must have $8\mid cd$ so, without
 loss of generality, we may assume that $2\mid c$ and $4\mid
 d$. Furthermore, $3 \mid (a-2)(a-4)cd$. We outline our choices of
 $h,h',f_1,f_2$ in Table \ref{tab:odd}. As before, one can verify that
 the corresponding sequence $g_1,g_2,f_1,f_2$ is regular.

 \begin{table}[htb]\centering
   $\begin{array}{|l|l|l|l|l|l|l|l|}
      \hline
      \multicolumn{4}{|c|}{\rm \bf Degrees} & \multicolumn{4}{|c|}{\rm\bf Symmetric\ Polynomials}\\
      \hline
      \multicolumn{1}{|c|}{a-2} & \multicolumn{1}{|c|}{a-4} & \multicolumn{1}{|c|}{c} & \multicolumn{1}{|c|}{d} 
                                                                                      & \multicolumn{1}{|c|}{h} & \multicolumn{1}{|c|}{h'} & \multicolumn{1}{|c|}{f_1} & \multicolumn{1}{|c|}{f_2}\\ \hline
      3\alpha &3\alpha-2 & 2\gamma & 4\delta & e_3^\alpha & e_2^{(3\alpha-3)/2} e_1 & (e_1^2 + e_2)^\gamma & e_4^\delta\\
      3\alpha+2 &3\alpha & 2\gamma & 4\delta & e_2^{(3\alpha -1)/2} e_1 & e_3^\alpha & (e_1^2 + e_2)^\gamma & e_4^\delta\\
      4\alpha-1 &4\alpha-3 & 6\gamma & 4\delta & e_2^{2\alpha-2} e_3 & e_4^{\alpha-1}e_1 & (e_2^3 + e_3^2)^{\gamma} & (e_1^4+e_4)^{\delta}\\
      4\alpha-1 &4\alpha-3 & 2\gamma & 12\delta & e_4^{\alpha-1} e_3 & e_2^{2\alpha-2}e_1 & (e_1^2 + e_2)^{\gamma} & (e_3^4+e_4^3)^{\delta}\\
      4\alpha+1 &4\alpha-1 & 6\gamma & 4\delta & e_4^\alpha e_1 & e_2^{2\alpha-2}e_3 & (e_2^3 + e_3^2)^{\gamma} & (e_1^4+e_4)^{\delta}\\
      4\alpha+1 &4\alpha-1 & 2\gamma & 12\delta & e_2^{2\alpha}e_1 & e_4^{\alpha-1}e_3 & (e_1^2 + e_2)^{\gamma} & (e_3^4+e_4^3)^{\delta}\\
      \hline
    \end{array}$\\
    \bigskip
    \caption{$a$ odd}
    \label{tab:odd}
  \end{table}
\end{proof}

\begin{remark}\label{rem:reg_seq_proof}
  For each line in Table \ref{tab:even} and Table \ref{tab:odd}, one
  can prove that the polynomials $h,h',f_1,f_2$ form a regular
  sequence using Lemma \ref{lem:reg_seq_factors} and \cite[Cor.~17.8
  a]{Eisenbud}. As an example, we show that the polynomials in the
  third line of Table \ref{tab:even}, specifically
  \begin{equation*}
    (e_2^3+e_3^2)e_4^{3\alpha-1}, \quad (e_2^6+e_3^4+e_4^3)^\alpha,
    \quad e_2^p e_3^q, \quad e_1^d,
  \end{equation*}
  form a regular sequence.

  By Lemma \ref{lem:reg_seq_factors} and \cite[Cor.~17.8 a]{Eisenbud},
  it is enough to show that the sequences
  \begin{itemize}
  \item $e_2^3+e_3^2, e_2^6+e_3^4+e_4^3, e_2, e_1$
  \item $e_2^3+e_3^2, e_2^6+e_3^4+e_4^3, e_3, e_1$
  \item $e_4, e_2^6+e_3^4+e_4^3, e_2, e_1$
  \item $e_4, e_2^6+e_3^4+e_4^3, e_3, e_1$
  \end{itemize}
  are regular.

  Let us show the first sequence is regular. The ideal it generates is
  equal to $e_3^2,e_3^4+e_4^3, e_2, e_1$, therefore it suffices to
  show that these generators form a regular sequence. Using
  \cite[Cor.~17.8 a]{Eisenbud} again, it is enough to prove that
  $e_3,e_3^4+e_4^3, e_2, e_1$ is regular. Because of the ideal
  equality $(e_3,e_3^4+e_4^3, e_2, e_1)=(e_3,e_4^3, e_2, e_1)$, we
  only need to prove that $e_3,e_4^3, e_2, e_1$ is regular. This
  follows immediately from \cite[Cor.~17.8 a]{Eisenbud} and the fact
  that the elementary symmetric polynomials form a regular sequence.

  The other sequences are handled similarly.
\end{remark}

In summary, we have the following result.
 
\begin{theorem}
  There exists a regular sequence of type
  $S^{(2,2)} \oplus S^{(4)} \oplus S^{(4)}$ and degrees $(a,a,c,d)$ if
  and only if
  \begin{enumerate}
  \item $a = 2$ or $4$ and $(c,d) \neq (1,1)$, or
  \item $a\geq 5$ and $(a-2,a-4,c,d)$ is a regular degree sequence.
  \end{enumerate}
\end{theorem}

\section{Acknowledgements}
While we were working on the research for this article A.V.~(Tony)
Geramita became ill and was confined to hospital in Vancouver B.C.~in
late December 2015.  During this period we continued working, with Tony
contributing comments and suggestions by email as his health
permitted.  The proofs of the results were mostly completed by March
2016 and Tony sent his last comments on April 26.  Unfortunately Tony
passed away in Kingston on June 22, 2016.  He was a wonderful
mathematician, colleague and friend whom we miss greatly.

The authors gratefully acknowledge the partial support of NSERC for
this work.

\begin{appendix}
\section{Macaulay2 code}
\label{sec:macaulay2-code}

We present here the Macaulay2 code used to produce the example in
Remark \ref{rem:M2_rem}.

\begin{verbatim}
needsPackage "Depth"
R=QQ[x_1..x_5]
e=apply(5,i->sum(apply(subsets(gens R,i+1),product)))
l=apply(4,i->x_(i+1)^6-x_5^6)
g=apply(4,i->sum(apply(4,j->e_(j+1)*(x_(i+1)^(4-j)-x_5^(4-j)))))
isRegularSequence(l|{e_0})
isRegularSequence(g|{e_0})
\end{verbatim}
\end{appendix}

%%Proper spelling of Dragomir Djokovic's name
\def\Dbar{\leavevmode\lower.6ex\hbox to 0pt{\hskip-.23ex
    \accent"16\hss}D}

\end{document}